\documentclass{style}
\makeatletter
\@namedef{subjclassname@2020}{%
  \textup{2020} Mathematics Subject Classification}
\makeatother

\allowdisplaybreaks

\theoremstyle{plain}
\newtheorem{theorem}{Theorem}[section]
\newtheorem{lemma}[theorem]{Lemma}

\newtheorem{proposition}[theorem]{Proposition}

\newtheorem{condition}[theorem]{Condition}

\theoremstyle{definition}
\newtheorem{definition}[theorem]{Definition}

\newtheorem{remark}[theorem]{Remark}

\numberwithin{equation}{section}

\makeatletter              
\let\c@equation\c@theorem  
\makeatother

\DeclareMathOperator{\End}{End}

\newcommand{\cA}{\mathcal{A}}
\newcommand{\C}{\mathcal{C}}
\newcommand{\B}{\mathcal{B}}
\newcommand{\D}{\mathcal{D}}
\newcommand{\E}{\mathcal{E}}
\newcommand{\Z}{\mathcal{Z}}
\newcommand{\M}{\mathcal{M}}
\newcommand{\N}{\mathcal{N}}

\newcommand{\dS}{\mathbb{S}}

\newcommand{\fp}{\mathfrak{p}}

\newcommand{\kk}{\Bbbk}
\newcommand{\id}{\textnormal{\textsf{Id}}}

\newcommand{\lv}{\hspace{.03cm}{}^{\vee} \hspace{-.03cm}}
\newcommand{\rv}{\hspace{.02cm}{}^{\vee} \hspace{-.05cm}}

\newcommand{\ev}{\textnormal{ev}}
\newcommand{\coev}{\textnormal{coev}}

\newcommand{\Hom}{{\sf Hom}}
\newcommand{\unit}{\mathds{1}}

\newcommand{\uNat}{\textnormal{\underline{Nat}}}

\newcommand{\Rex}{{\sf Rex}}

\newcommand{\ra}{{\sf ra}}
\newcommand{\la}{{\sf la}}

\newcommand{\tr}{\triangleright}

\newcommand{\rev}{\textnormal{rev}}
\newcommand{\op}{\textnormal{op}}

\newcommand{\mir}{\textnormal{mir}}

\definecolor{darkgreen}{RGB}{55,138,0}
\definecolor{burntorange}{RGB}{180,85,0}
\definecolor{navyblue}{RGB}{18,40,180}
\definecolor{cyan(process)}{rgb}{0.0, 0.6, 1.0}

\begin{document}
\allowdisplaybreaks

\begin{abstract}
    Let $U:\mathcal{C}\rightarrow\mathcal{D}$ be a strong monoidal functor between abelian monoidal categories admitting a right adjoint $R$, such that $R$ is exact, faithful and the adjunction $U\dashv R$ is coHopf. Building on the work of Balan \cite{balan2017hopf}, we show that $R$ is separable (resp., special) Frobenius monoidal if and only if $R(\mathds{1}_{\mathcal{D}})$ is a separable (resp., special) Frobenius algebra in $\mathcal{C}$. If further, $\mathcal{C},\mathcal{D}$ are pivotal (resp., ribbon) categories and $U$ is a pivotal (resp., braided pivotal) functor, then $R$ is a pivotal (resp., ribbon) functor if and only if $R(\mathds{1}_{\mathcal{D}})$ is a symmetric Frobenius algebra in $\mathcal{C}$. As an application, we construct Frobenius monoidal functors going into the Drinfeld center $\mathcal{Z}(\mathcal{C})$, thereby producing Frobenius algebras in it.

\end{abstract}


\maketitle

\setcounter{tocdepth}{1}


\section{Introduction}
Frobenius algebras are important mathematical objects with connections to many different areas of research, including Topological Quantum Field Theories (TQFTs) \cite{kock2004frobenius,carqueville2018line}, Conformal Field Theories (CFTs) \cite{fuchs2002tft} and Computer Science \cite{coecke2011interacting}.
One way of producing new Frobenius algebras in a monoidal category $\C$ from known Frobenius algebras in another such category $\D$ is via Frobenius monoidal functors $F:\D\rightarrow\C$ \cite{day2008note}. Thus, our goal is to construct Frobenius monoidal functors, especially with properties like being braided, ribbon, pivotal, separable, special,  etc., that in turn yield vital properties for Frobenius algebras used in applications.
To accomplish our goal, we generalize results from \cite{balan2017hopf} which established necessary conditions under which the right adjoint in a coHopf adjunction $U\dashv R$ is a Frobenius monoidal functor. 
Our main result is the following.

\begin{theorem}\label{thm:intro1}
Let $U:\C\rightarrow\D$ be a strong monoidal functor between abelian monoidal categories admitting a right adjoint $R$, such that $R$ is exact, faithful and the adjunction $U\dashv R$ is coHopf. Then, we get that $R$ is a $\circledast$ monoidal functor if and only if $R(\unit_{\D})$ is a $\divideontimes$ algebra in $\C$, as summarized in the table below.\vspace*{0.1cm}
\begin{center}
{\small
\begin{tabular}{ |c|c|c|c|c|c| }
\hline
\multirow{2}{*}{Reference} & \multicolumn{3}{c|}{Input} & $R(\unit_{\D})$ & $R$   \\
\cline{2-6}
& $\C$ & $\D$ & $U$ & $\divideontimes$ & $\circledast$ \\
\hline 
Thm.~\ref{thm:speFrobcoHopf} & $\otimes$ & $\otimes$ & strong $\otimes$ & separable Frob. & separable Frob. \\
Thm.~\ref{thm:speFrobcoHopf} & $\otimes$ & $\otimes$ & strong $\otimes$ & special Frob. & special Frob. \\
Thm.~\ref{thm:pivotalcoHopf} & pivotal & pivotal & pivotal, strong $\otimes$ &  symmetric Frob. & pivotal Frob. \\
Thm.~\ref{thm:ribboncoHopf} & ribbon & ribbon  & ribbon, strong $\otimes$ & symmetric Frob. & ribbon Frob. \\
\hline
\end{tabular}
}
\end{center}
\end{theorem}

For applications to TQFTs and CFTs, one needs to construct Frobenius algebras in Modular Tensor Categories (MTCs). In particular, Frobenius algebras in MTCs have been used to construct new MTCs via the category of local modules \cite{pareigis1995braiding,schauenburg2001monoidal,kirillov2002q,laugwitz2022constructing}. The prototypical example of a MTC is the Drinfeld center $\Z(\C)$ of a spherical finite tensor category $\C$ \cite{muger2003subfactors,shimizu2017ribbon}. 
Thus, as an application, we apply Theorem~\ref{thm:intro1} to a canonical functor $\Psi$ from $\Z(\C)$ to a certain category of endofunctors of a module category $\M$ over $\C$. This yields a Frobenius monoidal functor $\Psi^{\ra}$ (the right adjoint of $\Psi$) going into $\Z(\C)$, thereby producing Frobenius algebras in $\Z(\C)$.

\begin{theorem}\label{thm:intro2}
    Let $\C$ be a finite tensor category and $\M$ be an indecomposable, exact left $\C$-module category. Then the following statements hold.
    \begin{enumerate}
        \item[\upshape{(i)}] $\Psi^{\ra}$ is a (resp., separable, special) Frobenius monoidal functor if and only if $\Psi^{\ra}(\id_{\M})$ is (resp., separable, special) Frobenius algebra in $\Z(\C)$. 
        \item[\upshape{(ii)}] If $\C$ is pivotal and $\M$ a pivotal left $\C$-module category, then $\Psi^{\ra}$ is pivotal if and only if $\Psi^{\ra}(\id_{\M})$ is a symmetric Frobenius algebra in $\Z(\C)$. 
    \end{enumerate}
\end{theorem}

The question of when the algebra $\Psi^{\ra}(\id_{\M})$ is (resp., symmetric, special, separable) Frobenius is addressed in the work \cite{yadav2023unimodular}.

\begin{remark}
The algebra $\Psi^{\ra}(\id_{\M})$ is isomorphic to the algebras $\uNat(\id_{\M},\id_{\M})$ appearing in \cite{fuchs2021internal}. In future works, we hope to clarify this connection further.
\end{remark}

This paper is organized as follows. We review background material on monoidal categories in Section~\ref{sec:background}. In Sections~\ref{sec:coHopf} and \ref{sec:application}, Theorem~\ref{thm:intro1} and Theorem~\ref{thm:intro2} are proved, respectively.


\subsection*{Acknowledgements}
The author would like to thank Chelsea Walton for her guidance and support throughout the project. The author would also like to thank the anonymous referee for providing detailed comments that significantly improved the exposition of our manuscript. The author is partially supported by Nettie S.~Autrie Research Fellowship from Rice University.


\section{Background on monoidal categories}\label{sec:background}
In this work, $\kk$ will denote an algebraically closed field. We discuss monoidal categories and functors in Section~\ref{subsec:moncat}, pivotal categories and functors in Section~\ref{subsec:duality}, braided categories in Section~\ref{subsec:braided}, algebras in monoidal categories in Section~\ref{subsec:algebras}, adjunctions in Section~\ref{subsec:adjunctions}, and Drinfeld centers in Section~\ref{subsec:drinfeld}. 
For further details, see \cite{mac2013categories,etingof2016tensor,turaev2017monoidal}.


\subsection{Monoidal categories}\label{subsec:moncat}
A \textit{monoidal category} consists of a category $\C$ equipped with a functor $\otimes: \C\times\C\rightarrow \C$ (called the \textit{tensor product}), an object $\unit\in \C$ (called the \textit{unit object}) and natural isomorphisms
$X\otimes(Y \otimes Z) \cong (X\otimes Y)\otimes Z, \;  X\otimes \unit \cong X \cong \unit \otimes X$, for all $X,Y,Z\in \C$, 
satisfying the pentagon axiom and the triangle axiom. By Mac Lane's coherence theorem \cite[VII.2]{mac2013categories}, we can (and will) assume that the natural isomorphisms above are identities.

In the following, we will use the notation $(-)^{\op}$ to denote opposite of a category, functor or natural transformations. Also, let $\C^{\rev}$ denote the category $\C$ with the opposite tensor product $\otimes^{\rev}$, that is, $X\otimes^{\rev}Y :=Y\otimes X$. Then both $(\C^{\op},\otimes,\unit)$ and $(\C^{\rev},\otimes^{\rev},\unit)$ are monoidal categories.


\subsubsection*{Monoidal functors} Let $(\C,\otimes_{\C},\unit_{\C})$ and $(\D,\otimes_{\D},\unit_{\D})$ be two monoidal categories. 
\smallskip

A \textit{monoidal functor} from $\C$ to $\D$ is a tuple $(F,F_2,F_0)$ consisting of a functor $F:\C\rightarrow \D$, a natural transformation $F_2 = \{F_2(X,Y):F(X)\otimes_{\D} F(Y) \rightarrow F(X\otimes_{\C} Y)  \}_{X,Y\in \C}$ and a morphism $F_0:\unit_{\D} \rightarrow F(\unit_{\C}) $ in $\D$ such that the following conditions are satisfied for all $X,Y,Z \in \C$:
\begin{equation*}
\begin{split}
    F_2(X,Y\otimes_\C Z)\;(\id_{F(X)} \otimes_{\mathcal{D}} F_2(Y,Z))\; = \; F_2(X \otimes_\C Y,Z)(F_2(X,Y)\otimes_{\mathcal{D}}\id_{F(Z)}),\\
    \smallskip
    F_2(\unit_{\C},X) \; (F_0 \otimes_{\mathcal{D}} \id_{F(X)}) \; = \; \id_{F(X)},\hspace{0.3cm}  \hspace{0.3cm}
    F_2(X,\unit_\C) \;  (\id_{F(X)} \otimes_{\mathcal{D}} F_0) \;  = \; \id_{F(X)}. 
\end{split}
\end{equation*}
A monoidal functor $(F,F_2,F_0)$ is called \textit{strong} if $F_2$ and $F_0$ are isomorphisms in $\D$. If $F$ is strong monoidal and an equivalence, we call it a \textit{monoidal equivalence}.
\smallskip

A \textit{comonoidal functor} from $\C$ to $\D$ is a tuple $(F,F^2,F^0)$ consisting of a functor $F:\C\rightarrow \D$, a natural transformation $F^2 = \{F^2(X,Y): F(X\otimes_{\C} Y)  \rightarrow F(X)\otimes_{\D} F(Y) \}_{X,Y\in \C}$ and a morphism $F_0: F(\unit_{\C}) \rightarrow \unit_{\D} $ such that $(F^{\op},(F^2)^{\op},(F^0)^{\op}): \C^{\op}\rightarrow \D^{\op}$ is a monoidal functor.
\smallskip

A \textit{Frobenius monoidal functor} \cite[Definition~1]{day2008note} is defined as a tuple $(F,F_0,F_2,F^2,F^0)$ where $(F,F_2,F_0):\C\rightarrow \D$ is a monoidal functor, $(F,F^2,F^0):\C\rightarrow \D$ is a comonoidal functor and for all $X,Y,Z \in \C$, the following holds:
\[
\begin{array}{c}
\smallskip
(\id_{F(X)}\otimes_{\D} F_2(Y,Z) )\;(F^2(X,Y)\otimes_{\D} \id_{F(Z)})  \; =\; F^2(X,Y\otimes_{\C} Z) F_2(X\otimes_{\C} Y,Z), \\
(F_2(X,Y)\otimes_{\D} \id_{F(Z)})\; ( \id_{F(X)}\otimes_{\D} F^2(Y,Z)) \; = \;F^2(X\otimes_{\C} Y, Z) F_2(X, Y \otimes_{\C} Z). 
\end{array}
\]
It is clear that, $F$ is Frobenius monoidal if and only if $F^{\op}$ is. Now, suppose that $\D$ is a $\kk$-linear monoidal category. Then, we call a Frobenius monoidal functor $F:\C\rightarrow\D$ \textit{separable} if it satisfies $F_2(X,Y)\circ F^2(X,Y) = \beta_{2} \; \id_{F(X\otimes Y)}$ for some $\beta_2\in \kk^{\times}$. 
If in addition, $F^0\circ F_0=\beta_0\; \id_{\unit}$ holds for some $\beta_0\in \kk^{\times}$, we call it \textit{special}.


\subsubsection*{Monoidal natural transformation}
Let $F,G: \C\rightarrow\D$ be monoidal functors. A \textit{monoidal natural transformation} $\alpha= \{\alpha_X:F(X)\rightarrow G(X) \}_{X\in \C}$ is a natural transformation such that the following conditions are satisfied
\begin{equation*}
    G_2(X,Y)\circ (\alpha_X\otimes_{\D} \alpha_{Y}) = \alpha_{X\otimes_{\C}Y}  \circ F_2(X,Y), \hspace{1cm} \alpha_{\unit}\circ F_0 = G_0.
\end{equation*} 

Let $F,G: \C\rightarrow\D$ be comonoidal functors. A \textit{comonoidal natural transformation} $\alpha:F \rightarrow G$ is a natural transformation $\alpha$ such that $\alpha^{\op}$ (defined as $\alpha^{\op}_X=(\alpha_X)^{\op}$) is a monoidal natural transformation.


\subsection{Duality in monoidal categories}\label{subsec:duality}
We call a monoidal category \textit{rigid} if every object $X\in \C$ comes equipped with a left dual and a right dual, i.e., there exists an object $\lv X$ (\textit{left dual}) along with co/evaluation maps $\ev_X:\lv X\otimes X\rightarrow \unit$, $\coev_X:\unit\rightarrow X\otimes \lv X$ and an object $X\rv$ (\textit{right dual}) with co/evaluation maps $\widetilde{\ev}_X:X\otimes X\rv\rightarrow \unit$, $\widetilde{\coev}_X: \unit\rightarrow X\rv \otimes X$ satisfying the usual snake relations.

The maps $X\mapsto \lv X$ and $X \mapsto X\rv$ extend to monoidal equivalences from $\C^{\rev}$ to $\C^{\op}$. 
We can (and will) replace $\C$ by an equivalent monoidal category and choose duals in a suitable way to ensure that $\lv(-)$ and $(-)\rv$ are strict monoidal and mutually inverse to each other (see \cite[Lemma~5.4]{shimizu2015pivotal}). 

\subsubsection*{Frobenius monoidal functors and duality}
Suppose that we have a Frobenius monoidal functor $(F,F_2,F_0,F^2,F^0)$ between rigid categories $\C,\; \D$. Then by \cite[Theorem~2]{day2008note}, 
\begin{equation*}
    (F(\lv X),\;\; \overline{\ev}_{F(X)}= F^0\circ F(\ev_X) \circ F_2(\lv X,X),\;\; \overline{\coev}_{F(X)} = F^2(X,\lv X) \circ F(\coev_X) \circ F_0  )
\end{equation*} 
\noindent
\hspace{-0.17cm}is a left dual of $F(X)$ for $X\in\C$.
Thus, by uniqueness of dual objects, we have a unique family of natural isomorphisms $\zeta^F_X: F(\lv X) \rightarrow \lv F(X)$ called the \textit{duality transformation} of $F$. Explicitly, $\zeta^F_X$ and its inverse are given by
\begin{equation}\label{eq:zeta}
\begin{split}
    \zeta^F_X & = (\overline{\ev}_{F(X)}\otimes \id_{\lv F(X)}) \circ
    (\id_{F(\lv X)}\otimes \coev_{F(X)}) , \\
    \smallskip
    (\zeta_X^F)^{-1} & =  (\ev_{F(X)} \otimes F(\lv X) ) \circ
    ( \id_{\lv F(X)}\otimes \overline{\coev}_{F(X)} ).
\end{split}
\end{equation}
Also, define the natural isomorphism 
\begin{equation}\label{eq:xi}
    \xi^F_X:= \lv((\zeta^F_X)^{-1}) \circ \zeta^F_{\lv X} : F(\lv\lv X) \rightarrow \lv\lv F(X). 
\end{equation}

\begin{remark}
The above discussion follows \cite[Section~3.1]{shimizu2015pivotal}, \cite[Section~1]{ng2007higher} where these results are stated for $F$ being a strong monoidal functor. However, similar arguments apply in the case when $F$ is Frobenius monoidal.
\end{remark}

\begin{lemma}\label{lem:pivShimizuA}
Let $\C\xrightarrow{F}\D\xrightarrow{G}\E$ be a sequence of Frobenius monoidal functors between rigid monoidal categories. Then, we have that 
\begin{equation}\label{eq:zetaXiComp}
\zeta^{GF}_X = \zeta^G_{F(X)}\circ G(\zeta^F_X) \;\;\;\; \text{and} \;\;\;\; \xi^{GF}_X = \xi^G_{F(X)}\circ G(\xi^F_X).
\end{equation}
\end{lemma}
\begin{proof}
The same argument as \cite[Lemma~3.1]{shimizu2015pivotal} applies.
\end{proof}


\subsubsection*{Pivotal categories and functors}
A monoidal category with left duals, equipped with a natural isomorphism $\fp=\{\fp_X:X\rightarrow \lv\lv X\}_{X\in\C}$ satisfying $\fp_{X\otimes Y} = \fp_X\otimes \fp_Y$ for all $X,Y\in\C$
is called \textit{pivotal}. In a pivotal category, for each object $X$, $\lv X$ is also a right dual to $X$ with co/evaluation maps 
\begin{equation*}
    \widetilde{\coev}_X := (\id_{\lv X}\otimes \fp_X^{-1})\coev_{\lv X} \;\;\;\; \text{and} \;\;\;\; \widetilde{\ev}_X := \ev_{\lv X} (\fp_X \otimes \id_{\lv X}).
\end{equation*}
We will use this right dual when working with pivotal categories.

\smallskip

In \cite{ng2007higher}, the notion of a strong monoidal functor preserving pivotal structure was introduced. Generalizing it, we give the following definition.

\begin{definition}\label{defn:pivotalFunctor}
A Frobenius monoidal functor $F:\C\rightarrow\D$ between pivotal categories is said to \textit{pivotal} if it satisfies 
\begin{equation}\label{eq:pivotalFunctor}
    \fp^{\D}_{F(X)} = \xi^F_X \circ F(\fp^{\C}_X)\hspace{1cm} (X\in\C). 
\end{equation} 
\end{definition}

\begin{lemma}\label{lem:pivShimizuB}
Let $\C\xrightarrow{F}\D\xrightarrow{G}\E$ be a sequence of Frobenius monoidal functors between rigid monoidal categories. If $ \C,\; \D, \; \E$ are pivotal and $F,\; G $ are pivotal functors, then so is their composition $G\circ F$.
\end{lemma}
\begin{proof}
By, \cite[Proposition~4]{day2008note} $G\circ F$ is Frobenius monoidal. Also,
\[\fp^{\E}_{GF(X)} 
\stackrel{(\textnormal{\ref{eq:pivotalFunctor}})}{=}  
\xi^G_{F(X)} \circ G(\fp^{\D}_{F(X)}) 
\stackrel{(\textnormal{\ref{eq:pivotalFunctor}})}{=} 
\xi^G_{F(X)} \circ G(\xi^F_X \circ F(\fp^{\C}_X)) 
\stackrel{(\textnormal{\ref{eq:zetaXiComp}})}{=} \xi^{GF}_X \circ GF(\fp^{\C}_X).\]
Hence, $GF$ is a pivotal Frobenius functor.
\end{proof}


\subsection{Braided and ribbon categories}\label{subsec:braided}

A \textit{braided (monoidal) category} is a monoidal category $(\C,\otimes,\unit)$ equipped with a natural isomorphism $c=\{c_{X,Y}:X\otimes Y\rightarrow Y\otimes X \}_{X,Y\in\C}$ (called \textit{braiding}) satisfying the hexagon axiom. 
The \textit{mirror} of a braiding on $\C$ is defined by $c'_{X,Y}:=c_{Y,X}^{-1}$. We will let $\C^{\mir}$ denote the braided monoidal category $\C$ equipped with the mirror braiding $c'$. 
\smallskip

A \textit{ribbon category} $\C$ is a braided pivotal category such that the left and right twists coincide, that is, $\theta^l_X = \theta^r_X \; (:=\theta_X)$ for any object $X$ of $C$, where
\begin{align*}
\theta^l_{X}:=(\ev_X\otimes \id_X)(\id_{\lv X}\otimes c_{X,X})(\widetilde{\coev}_X\otimes \id_X), \\
\theta^r_X:=(\id_X\otimes \widetilde{\ev}_X)(c_{X,X}\otimes \id_{\lv X})(\id_X\otimes \coev_X).
\end{align*}

\subsubsection*{Braided and ribbon functors}
A monoidal functor $(F,F_2,F_0)$ between braided categories $(\C,c)$ and $(\D,d)$ is called \textit{braided} if it satisfies
\[ F_2(Y,X)\circ d_{F(X),F(Y)} = F(c_{X,Y})\circ F_2(X,Y) \hspace{1cm} (X,Y\in\C) .\]
A braided functor that is also a monoidal equivalence is called a \textit{braided equivalence}.
Also, a comonoidal functor $(F,F^2,F^0)$ between braided categories $(\C,c)$ and $(\D,d)$ is called \textit{cobraided} if $(F^{\op},(F^{2})^{\op}, (F^{0})^{\op})$ is braided. 

\begin{lemma}\label{lem:coBraid}
A braided strong monoidal functor is cobraided as well. \qed
\end{lemma}

A braided functor $F:\C\rightarrow\D$ between ribbon categories is called a \textit{ribbon functor} if it preserves twists, that is, $F(\theta^{\C}_X)=\theta^{\D}_{F(X)}$ for all $X$ in $\C$. We will need the following result later.

\begin{proposition}\cite[Proposition~4.4]{mulevicius2022condensation}\label{prop:brPivotalFunctor}
A braided Frobenius functor between ribbon categories is ribbon if and only if it is pivotal and cobraided. \qed
\end{proposition}


\subsection{Algebras in monoidal categories}\label{subsec:algebras}
Take $\C$ to be a monoidal category.
\smallskip

An \textit{algebra} in $\C$ is a triple $(A,m,u)$ consisting of an object  $A\in \C$, and morphisms $m:A\otimes A\rightarrow A$, $u:\unit\rightarrow A$  in $\C$  satisfying associativity and unitality constraints. A \textit{coalgebra} in $\C$ is a triple $(A,\Delta,\nu)$ consisting of an object  $A \in \C$, and morphisms $\Delta: A \rightarrow A \otimes A$, $\varepsilon: A \to \unit$  in $\C$ satisfying coassociativity and counitality constraints.

\subsubsection*{Frobenius algebras}
A \textit{Frobenius algebra} in $\C$ is a 5-tuple $(A,m,u,\Delta,\nu)$ where $(A,m,u)$ is an algebra and $(A,\Delta,\nu)$ is a coalgebra so that 
$(m \otimes \text{id}_A)(\text{id}_A \otimes \Delta) = \Delta m = (\text{id}_A \otimes m)(\Delta \otimes \text{id}_A)$ holds. 

Now, suppose that $\C$ is a $\kk$-linear monoidal category and $(A,m,u,\Delta,\nu)$ is a Frobenius algebra in $\C$. If $m\; \Delta = \beta_A \; \id_A$ holds for some $\beta_A\in \kk^{\times}$, we call the Frobenius algebra \textit{separable}. If furthermore, $\nu\; u = \beta_{\unit} \; \id_{\unit}$ for some $\beta_{\unit}\in \kk^{\times}$, we call it \textit{special} \cite{frohlich2006correspondences}.


\subsubsection*{Symmetric Frobenius algebras}
A Frobenius algebra $(A,m,u,\Delta,\nu)$ in a pivotal monoidal category $\C$ is called \textit{symmetric} \cite[Definition~2.22]{frohlich2006correspondences} if it satisfies 
\begin{equation}\label{eq:symalg}
    (\nu \; m \otimes \id_{\lv A})(\id_A\otimes\coev_A) =   
    (\id_{\lv A} \otimes \nu \; m)(\id_{\lv A} \otimes \fp_A^{-1}\otimes \id_A)(\coev_{\lv A} \otimes \id_A) .
\end{equation}

\begin{lemma}\label{lem:symFrobalternate}
Let $(A,m,u,\Delta,\nu)$ be a Frobenius algebra in a pivotal category $\C$. Then $A$ is symmetric if and only if the following holds
\begin{equation*}
    (\ev_A\otimes \id_A)(\id_{\lv A} \otimes \Delta\; u) = 
    (\id_A\otimes \ev_{\lv A})(\id_A\otimes \fp_A \otimes \id_{\lv A})(\Delta\; u \otimes \id_{\lv A}).
\end{equation*}
\end{lemma}
\begin{proof}
Observe that, $(\id_A\otimes \ev_{\lv A})(\id_A\otimes \fp_A \otimes \id_{\lv A})(\Delta\; u \otimes \id_{\lv A})$ is the inverse of the morphism $(\id_A \otimes \nu \; m)(\id_{\lv A} \otimes \fp_A^{-1}\otimes \id_A)(\coev_{\lv A} \otimes \id_A)$. Similarly, $(\ev_A\otimes \id_A)(\id_{\lv A} \otimes \Delta\; u)$ is the inverse of $ (\nu \; m\otimes \id_{\lv A}) (\id_A\otimes \coev_A)$. Thus by (\ref{eq:symalg}), the claim follows.
\end{proof}


\subsubsection*{(Braided) commutative algebras}
Let $\C$ be a braided category. Then an algebra $(A,m,u)\in \C$ is called \textit{(braided) commutative} if it satisfies $m\; c_{A,A} = m$. A coalgebra $(A,\Delta,\nu)$ is called \textit{(braided) cocommutative} if it satisfies $c_{A,A}\,\Delta=\Delta$. 
It is straightforward to see that if $F:\C\rightarrow\D$ is a (co)braided (co)monoidal functor between braided categories and $A$ is (co)commutative (co)algebra in $\C$, then $F(A)$ is a (co)commutative (co)algebra in $\D$.

The following lemma summarizes the relationship between the various functors and Frobenius algebras that we have defined above.

\begin{lemma}\label{lem:mon-preserve}
Let $(F:\C\rightarrow\D,F_2,F_0,F^2,F^0)$ be a Frobenius monoidal functor between monoidal categories. Take $(A,m,u,\Delta,\varepsilon)$ a Frobenius algebra in $\C$. Then the following statements hold.
\begin{enumerate}
    \item[\upshape{(a)}] $(F(A),\; F(m)F_2(A,A),\; F(u)F_0,\; F^2(A,A)F(\Delta),\; F^0F(\nu)) $ is a Frobenius algebra in $\D$.
    \item[\upshape{(b)}] Suppose that $\D$ is $\kk$-linear. If $F$ is separable (resp., special) and $A$ is separable (resp., special) Frobenius algebra in $\C$, then $F(A)$ is a separable (resp., special) Frobenius algebra in $\D$.
    \item[\upshape{(c)}] If $\C,\D$ are pivotal categories, $F$ is a pivotal monoidal functor and $A$ is a symmetric Frobenius algebra, then $F(A)$ is symmetric Frobenius as well.
\end{enumerate}
\end{lemma}

\begin{proof}
Part (a) is \cite[Corollary~5]{day2008note}, part (b) is straightforward, and part (c) is \cite[Proposition~4.8]{mulevicius2022condensation}.
\end{proof}


\subsection{Adjunctions}\label{subsec:adjunctions}
An \textit{adjoint pair}  $(L,R,\eta,\varepsilon)$ is a pair of functors $L:\C\rightarrow\D$ and $R:\D\rightarrow\C$ between categories $\C$ and $\D$ along with natural transformations $\eta=\{\eta_X:X\rightarrow RL(X)\}_{X\in\C}$ and $\varepsilon = \{\varepsilon_{Y}:LR(Y)\rightarrow Y\}_{Y\in\D}$ such that the following relations hold:
\begin{equation}\label{eq:snake}
    R(\varepsilon_Y)\; \eta_{R(Y)} = \id_{R(Y)} \hspace{1cm} \varepsilon_{L(X)}\;L(\eta_X) = \id_{L(X)}.
\end{equation} 
In this situation, we call $L$ the \textit{left adjoint} of $R$ and denote it as $R^{\la}$. Similarly, we call $R$ the \textit{right adjoint} of $L$ and denote it as $L^{\ra}$. We will also use the notation $L\dashv R$ to mean the $L,R$ form an adjoint pair with $L$ as the left adjoint. Then, $L\dashv R$ if and only if $R^{\op}\dashv L^{\op}$. 
\smallskip

A \textit{(co)monoidal adjunction} \cite[Section~2.5]{bruguieres2011hopf} is an adjunction $L\dashv R$, between monoidal categories such that $L,R$ are (co)monoidal functors, and the unit, counit of the adjunction are (co)monoidal natural transformations.
A \textit{Hopf adjunction} \cite[Section~2.8]{bruguieres2011hopf} is a comonoidal adjunction $L\dashv R$ such that the following \textit{Hopf operator morphisms} are invertible. 
\begin{equation*}\label{eq:hopfadj}
    \begin{gathered}
    H^l_{X,Y}: L(X\otimes R(Y)) \xrightarrow{L^2(X,R(Y))} L(X)\otimes LR(Y) \xrightarrow{\id_{L(X)}\otimes \varepsilon_Y} L(X)\otimes Y , \\
    H^r_{Y,X}: L(R(Y)\otimes X) \xrightarrow{L^2(R(Y),X)} LR(Y)\otimes L(X) \xrightarrow{\varepsilon_Y\otimes\id_{L(X)}} Y\otimes L(X). 
    \end{gathered}
\end{equation*} 

However, we will need the following dual notion.
A \textit{coHopf adjunction} is a monoidal adjunction $L\dashv R$ such that the following \textit{coHopf operator morphisms} are invertible for all $X\in\C$ and $Y\in \D$.
\begin{equation}\label{eq:cohopfadj}
    \begin{gathered}
    h^l_{Y,X}: R(Y)\otimes X \xrightarrow{\id_{R(Y)}\otimes \eta_X} R(Y)\otimes RL(X) \xrightarrow{R_2(Y,L(X))} R(Y\otimes L(X)), \\
    h^r_{X,Y}: X\otimes R(Y) \xrightarrow{\eta_X\otimes \id_{R(Y)}} RL(X)\otimes R(Y) \xrightarrow{R_2(L(Y),X)} R(L(X)\otimes Y). 
    \end{gathered}
\end{equation} 
We get that $L\dashv R$ is a coHopf adjunction if and only if $R^{\op}\dashv L^{\op}$ is a Hopf adjunction. Now consider the following results.

\begin{theorem}\label{thm:docadj}
Let $L:\C\rightarrow\D$ be left adjoint to $R:\D\rightarrow\C$ with unit $\eta$ and counit $\varepsilon$. Then the following hold:
\begin{enumerate}
    \item[\upshape(a)] If $L$ is strong monoidal with structure maps $L^2,L^0$, then $R$ admits the following monoidal structure making the adjunction $L\dashv R$ monoidal:
    \begin{equation*}
        R_2(Y,Y') = R(\varepsilon_Y\otimes \varepsilon_{Y'})\circ R (L^2(R(Y),R(Y'))) \circ \eta_{R(Y)\otimes R(Y')} , \;\;\;  R_0 = R(L^0) \circ \eta_{\unit_{\C}} .
    \end{equation*}
    \item[\upshape(b)] If $\C,\D$ are rigid monoidal categories, then any comonoidal adjunction between them is Hopf and any monoidal adjunction is coHopf.
\end{enumerate} 
\end{theorem}
\begin{proof}
    Part (a) is a classical result of Kelly \cite[Section~2.1]{kelly1974doctrinal}. Part (b) follows from \cite[Proposition~3.5]{bruguieres2011hopf}.
\end{proof}


\subsection{Drinfeld centers}\label{subsec:drinfeld}
The \textit{Drinfeld center} of a monoidal category $\C$, denoted $\Z(\C)$, is defined as the category with objects as pairs $(X,\sigma)$, where $X$ is an object in $\C$, and $\sigma$ is a half-braiding on $X$. A \textit{half-braiding} on $X$ is a natural isomorphism $\sigma=\{\sigma_Y:Y\otimes X \rightarrow X\otimes Y\}_{Y\in \C}$ satisfying 
$\sigma_{Y\otimes Z} = (\sigma_Y\otimes \id_Z)(\id_Y \otimes \sigma_Z)$ for all $Y,Z\in \C$.
The morphisms $(X,\sigma)\rightarrow (Y,\sigma')$ are given by morphisms $f\in \Hom_{\C}(X,Y)$ satisfying 
$(f\otimes \id_Z)\sigma_Z = \sigma'_Z(\id_Z\otimes f)$.
The monoidal product is $(X,\sigma)\otimes (Y,\sigma') = (X\otimes Y, \gamma$) where $\gamma_Z:= (\id_X\otimes \sigma'_Z) (\sigma_Z\otimes\id_Y)$. We also have the forgetful functor $U_{\C}:\Z(\C)\rightarrow \C$ given by $(X,\sigma) \mapsto X$ which is strong monoidal.


\section{Frobenius monoidal functors from (co)Hopf adjunctions}\label{sec:coHopf}
The goal of this section is to construct separable, special, pivotal, ribbon Frobenius monoidal functors from coHopf adjunctions $U\dashv R$. To do this, we start by recalling results from \cite{balan2017hopf} which established sufficient conditions under which the right adjoint $R$ is Frobenius monoidal (see Theorem~\ref{thm:frobcoHopf}). 
We recall Balan's result and discuss some preliminaries in Section~\ref{subsec:3.prel}. We generalize to the separable and special Frobenius setting in Section~\ref{subsec:3.spesep}, to the pivotal setting in Section~\ref{subsec:3.pivotal}, and to the ribbon setting in Section~\ref{subsec:3.ribbon}.


\subsection{Preliminaries}\label{subsec:3.prel}
In the following, we will often replace $\otimes$ by $\cdot$, in order to fit equations. Equalities marked as $(N)$ will commute because of naturality. 
Consider the following conditions on a functor $U:\C\rightarrow\D$, which will be used throughout the rest of this section. 

\begin{condition}\label{cond:main}
$(U:\C\rightarrow\D,U^2,U^0)$ is a strong monoidal functor between abelian monoidal categories admitting a right adjoint $R$ (with unit $\eta^r$, counit $\varepsilon^r$) such that:
\begin{enumerate}
    \item[\upshape{(a)}] $U\dashv R$ is a coHopf adjunction,
    \item[\upshape{(b)}] $R$ is exact, and
    \item[\upshape{(c)}] $R$ is faithful.
\end{enumerate} 
\end{condition}

\begin{remark}\label{rem:coHopf}
Under these conditions, the functor $R$ is monoidal with structure maps $R_2,R_0$ in Theorem~\ref{thm:docadj}(a) making $U\dashv R$ a monoidal adjunction. Hence, condition \upshape{(a)} is satisfied when $\C,\D$ are rigid [Theorem~\ref{thm:docadj}(b)]. In other words, the left and right coHopf operators $h^l, h^r$ of (\ref{eq:cohopfadj}) are invertible in this case. 
\end{remark}

Now consider the following result connecting the Frobenius properties of the right adjoint $R$ and the algebra $R(\unit)$. 

\begin{theorem}\label{thm:frobcoHopf}
Suppose that Condition \ref{cond:main} is satisfied. Then, $R$ is Frobenius monoidal if and only if the algebra $R(\unit)$ is a Frobenius algebra in $\C$. 

In particular, when $R(\unit)$ is Frobenius with comultiplication $\Delta$ and counit $\nu$, then we get the following results: 
\begin{enumerate}
    \item[\upshape{(a)}]  $R$ is Frobenius monoidal with the following comonoidal constraints:
    \begin{equation}\label{eq:R^2}
    R^2(X,Y) = (h^l_{X,R(Y)})^{-1} \circ R(\id_X \otimes \eta^l_Y), \hspace{1cm} R^0 = \nu.
    \end{equation} 
    
    \item[\upshape{(b)}] $R\dashv U$ is a Hopf adjunction with the counit $\varepsilon^l$ given by
    \begin{equation*}
        \varepsilon^l(X): RU(X) = R(\unit\otimes U(X))\xrightarrow{(h^l_{\unit,X})^{-1}} R(\unit) \otimes X \xrightarrow{\nu\otimes \id_{X}} X,   
    \end{equation*}
    and the unit $\eta^l:X\rightarrow UR(X)$ as the unique morphism making the following equation hold:
    \begin{equation}\label{eq:etaL}
        \eta^l_X \circ \varepsilon^r_X = UR(\varepsilon^r_X) \circ \varepsilon^r_{URUR(X)} \circ U(h^l_{\unit,RUR(X)}) \circ U(\id_{R(\unit)}\cdot h^l_{\unit,R(X)}) \circ U(\Delta u\cdot \id_{R(X)}).
    \end{equation}
\end{enumerate}
\end{theorem}

Theorem~\ref{thm:frobcoHopf} is proved below after we prove a small lemma. Consider an adjoint pair $L\dashv L^{\ra}$ with $L:\cA\rightarrow\B$. Let $\varepsilon$ be the counit and $T=L^{\ra}L$ be the corresponding monad on $\cA$ \cite[Chapter~6]{mac2013categories}. Then, there is a canonical comparison functor $K:\B\rightarrow\cA^T$, where $\cA^T$ is the category of $T$-modules. The adjunction $L\dashv L^{\ra}$ is called \textit{premonadic} if the functor $K$ is full and faithful. Dually, for a functor $U$, an adjunction $U^{\la}\dashv U$ is called \textit{precomonadic} if $U^{\op}\dashv (U^{\la})^{\op}$ is premonadic.
By \cite[Corollary~3.9 and Theorem~3.11]{barr2000toposes}, the following are equivalent:
\begin{enumerate}
    \item $L\dashv L^{\ra}$ is precomonadic;
    
    \smallskip

    \item $\varepsilon_X$ is the cokernel of some parallel pair of morphisms $\forall$ $X\in\B$;
        
    \item $\begin{tikzcd}
        LL^{\ra}LL^{\ra}(X) \arrow[r, "LL^{\ra}(\varepsilon_X)", shift left] \arrow[r, "\varepsilon_{L L^{\ra}(X)}"', shift right] & L L^{\ra}(X) \arrow[r, "\varepsilon_X"] & X,
    \end{tikzcd} \text{ is a coequalizer $\forall$ $X\in \B$.}$ \hfill\refstepcounter{equation}\textup{(\theequation)}\label{eq:coequalizer}
    
\end{enumerate}
The following lemma is probably well-known, but we could not find a proof.

\begin{lemma}\label{lem:premonadic}
Let $L\dashv R:\B\rightarrow\cA$  be an adjunction between abelian categories with counit $\varepsilon$. Then, the following are equivalent:
\begin{align*}
    L\dashv R\; \text{is precomonadic}  \; \iff \; \varepsilon_X \; \textit{is epic} \; \forall\; X\in\B  \; \iff \; R\; \text{is faithful}.
\end{align*}
\end{lemma}
\begin{proof}
By the definition of an abelian category, a morphism $f$ in it is an epimorphism if and only if it is a cokernel of some parallel pair of morphisms. Thus, using the equivalent definition of being precomonadic given in (\ref{eq:coequalizer}), we get that $\varepsilon_X$ is an epimorphism if and only if $L\dashv R$ is precomonadic. Lastly, $\varepsilon_X$ is epic if and only if $R$ is faithful \cite[Chapter~4]{mac2013categories}. Thus, the proof is finished.  
\end{proof}

\begin{proof}[Proof of Theorem~\ref{thm:frobcoHopf}]
($\Rightarrow$): Suppose that $R$ is Frobenius monoidal. Then, since $\unit$ is Frobenius, $R(\unit)$ is Frobenius by Lemma~\ref{lem:mon-preserve}(a).
\smallskip

\noindent
($\Leftarrow$):
By Condition~\ref{cond:main}(b), $U\dashv R$ is a coHopf adjunction, which implies that $R^{\op}\dashv U^{\op}$ is Hopf adjunction. Also, $R(\unit)$ is a Frobenius algebra in $\C$; this implies that $R^{\op}(\unit)$ is a Frobenius algebra in $\C^{\op}$. As $R$ is faithful, by Lemma~\ref{lem:premonadic}, we get that $U\dashv R$ is premonadic. Hence, $R^{\op}\dashv U^{\op}$ is precomonadic.  Thus, we can apply \cite[Proposition~4.5]{balan2017hopf} to the adjunction $R^{\op}\dashv U^{\op}$ to get that $R^{\op}$ is Frobenius monoidal. Hence, $R$ is Frobenius monoidal. Similarly, by \cite[Theorem~4.4]{balan2017hopf}, $R\dashv U$ is a Hopf adjunction. Finally, the claim about the expressions for $\varepsilon^l,\eta^l,R^2,R^0$ follows by translating the expressions in \cite{balan2017hopf} into our setting.
\end{proof}

The next properties of the coHopf operators $h^l$ of (\ref{eq:cohopfadj}) will be used in the following sections.

\begin{lemma}\label{lem:operatorCoHopf}
The coHopf operators $h^l$ satisfy the following relations:
\begin{enumerate}
    \item[\upshape{(a)}] \hspace{-0.075cm}$R_2(X,Y\otimes UR(Z)) \circ (\id_{R(X)} \otimes h^l_{Y,R(Z)}) = h^l_{X\otimes Y,R(Z)} \circ (R_2(X,Y)\otimes \id_{R(Z)})$,
    \smallskip
    \item[\upshape{(b)}] $\varepsilon^r_{X\otimes UR(Y)}\circ U(h^l_{X,R(Y)}) = (\varepsilon^r_{X} \otimes \id_{UR(Y)}) \circ U_2^{-1}(R(X),R(Y)),$ 
    \smallskip
    \item[\upshape{(c)}] $R(\id_X\otimes \varepsilon^r_Y) \circ h^l_{X,R(Y)} = R_2(X,Y),$
    \smallskip
    \item[\upshape{(d)}] $R(\id_X\otimes U_2^{-1}(Y,Z))\circ h^l_{X,Y\otimes Z} = h^l_{X\otimes U(Y),Z} \circ (h^l_{X,Y} \otimes \id_{Z}).$
\end{enumerate}
\end{lemma}
\begin{proof}
We prove part \upshape{(a)} here, parts \upshape{(b)}, \upshape{(c)} and \upshape{(d)} are proved in a similar manner. \vspace{0.15cm}
{\[ 
\begin{array}{rcl}
    \text{LHS} &=& R_2(X,Y\otimes UR(Z)) \circ (\id_{R(X)} \otimes h^l_{Y,R(Z)})\\
    & \stackrel{(\textnormal{\ref{eq:cohopfadj}})}{=}&
    R_2(X,Y\otimes UR(Z)) \circ  (\id_{R(X)} \otimes R_2(Y,UR(Z)) ) \circ (\id_{R(X)\otimes R(Y)} \otimes \eta^r_{R(Z)}) \\
    & \stackrel{(\heartsuit)}{=}&
    R_2(X\otimes Y,UR(Z)) \circ (R_2(X,Y)\otimes \id_{RUR(Z)})  \circ (\id_{R(X)\otimes R(Y)} \otimes \eta^r_{R(Z)}) \\
    & \stackrel{(N)}{=} & 
    R_2(X\otimes Y,UR(Z)) \circ (\id_{R(X\otimes Y)} \otimes \eta^r_{R(Z)}) \circ (R_2(X,Y)\otimes \id_{R(Z)}) \\
    & \stackrel{(\textnormal{\ref{eq:cohopfadj}})}{=} &
    h^l_{X\otimes Y,R(Z)} \circ (R_2(X,Y)\otimes \id_{R(Z)}).
\end{array}
\]}
Here, the equality $(\heartsuit)$ holds because $R$ is a monoidal functor.
\end{proof}

\begin{proposition}\label{prop:R1FrobinZC}
Suppose that Condition~(\ref{cond:main}) is satisfied and $R(\unit)$ is a Frobenius algebra in $\C$ with comultiplication $\Delta$ and counit $\nu$. Then, 
$((R(\unit),\sigma),\, R_2(\unit,\unit),\, R_0,$ $\Delta,$ $\nu)$
is a Frobenius algebra in $\Z(\C)$. Here $\sigma_X=(h^l_{\unit,X})^{-1} h^r_{X,\unit}$ is the half-braiding of the object $R(\unit)$.
\end{proposition}
\begin{proof}
By assumption, $U\dashv R$ is a coHopf adjunction. Thus, by applying \cite[Corollary~6.7]{bruguieres2011hopf} to the Hopf adjunction $R^{\op}\dashv U^{\op}$, we get that $((R(\unit),\sigma),R_2(\unit,\unit),R_0)$ is an algebra in $\Z(\C)$ where $\sigma_X=(h^l_{\unit,X})^{-1}  h^r_{X,\unit}$. 

Also, by Theorem~\ref{thm:frobcoHopf}, $R\dashv U$ is a Hopf adjunction. Thus, by \cite[Corollary~6.7]{bruguieres2011hopf}, we have that $((R(\unit),\rho),R^2(\unit,\unit),R^0)$ is a coalgebra in $\Z(\C)$ where $\rho_X= H^l_{\unit,X} (H^r_{X,\unit})^{-1}$. By \cite[Remark~4.3(3)]{balan2017hopf}, $H^r_{X,Y}=(h^r_{X,Y})^{-1}$ and $H^l_{X,Y} = (h^l_{X,Y})^{-1}$ for all $X,Y\in \C$. Thus, it is clear that $\sigma_X=\rho_X$ for all $X\in \C$.

By Theorem~\ref{thm:frobcoHopf}, we know that $R^0=\nu$. Next we will show that $R^2(\unit,\unit)=\Delta$.
By (\ref{eq:R^2}), $R^2(\unit,\unit) \circ R(\varepsilon^r_X) = (h^l_{\unit,R(\unit)})^{-1}\circ R(\eta^l_X \circ \varepsilon^r_X)$. To show that $R^2(\unit,\unit)=\Delta$, it suffices to show that $h^l_{\unit,R(\unit)}\circ \Delta \circ R(\varepsilon^r_X) = R(\eta^l_X \circ \varepsilon^r_X)$ because $R(\varepsilon^r_X)$ is epic. Set $m=R_2(\unit,\unit)$, $u=R_0$ and consider the following commutative diagram.
{\small
\[\begin{tikzcd}[column sep=1.9em]
    {RUR(\unit)} && {R(\unit) R(\unit)} && {R(\unit) R(\unit)} & {R(\unit)} \\
	{RU(R(\unit) R(\unit))} && {R(\unit) R(\unit) R(\unit)} && {R(\unit) R(\unit)} \\
	\\
	{RU(R(\unit) R(\unit) R(\unit))\;} && {\; R(\unit) R(\unit) R(\unit) R(\unit)} && {R(\unit) R(\unit) R(\unit)} & {R(\unit) R(\unit)} \\
	\\
	{RU(R(\unit) RUR(\unit))} && {RU(R(\unit) R(\unit))} && {RUR(\unit) R(\unit)} & {R(\unit) R(\unit)} \\
	\\
	{RU(R(\unit) RUR(\unit))} && {R(UR(\unit) URUR(\unit))} && {R(UR(\unit) UR(\unit))} & {RUR(\unit)} \\
	{RURURUR(\unit)} && {RURUR(\unit)} &&& {RUR(\unit)}
	\arrow["{RU(u\cdot\id)}"', from=1-1, to=2-1]
	\arrow["{RU(\Delta\cdot\id)}"', from=2-1, to=4-1]
	\arrow["{RU(\id\cdot h^l_{\unit,R(\unit)})}"', from=4-1, to=6-1]
	\arrow["{R(\varepsilon^r_{URUR(\unit)})}"', from=9-1, to=9-3]
	\arrow["{RUR(\varepsilon^r_{\unit})}"', from=9-3, to=9-6]
	\arrow["{R(\varepsilon^r_{\unit})}"{description}, shift left=1, curve={height=-30pt}, from=1-1, to=1-6]
	\arrow["{(h^l_{\unit,R(\unit)})^{-1}}", from=1-1, to=1-3]
	\arrow["\Delta", from=1-6, to=4-6]
	\arrow["{(h^l_{\unit,R(\unit)R(\unit)})^{-1}}"', from=2-1, to=2-3]
	\arrow["{(h^l_{\unit,R(\unit)R(\unit)R(\unit)})^{-1}}", from=4-1, to=4-3]
	\arrow["\id\cdot\Delta\cdot\id", from=2-3, to=4-3]
	\arrow["{\id \cdot u\cdot \id}", from=1-3, to=2-3]
	\arrow[""{name=0, anchor=center, inner sep=0}, "{RU(\id\cdot m)}", from=4-1, to=6-3]
	\arrow["{RU(\id \cdot R(\varepsilon^r_{\unit}))}"', from=6-1, to=6-3]
	\arrow[Rightarrow, no head, from=1-3, to=1-5]
	\arrow["m", from=1-5, to=1-6]
	\arrow["{h^l_{\unit,R(\unit)}\cdot\id}"{description}, from=4-5, to=6-5]
	\arrow["{h^l_{UR(\unit),R(\unit)}}"{description}, from=6-5, to=8-5]
	\arrow["{R(\varepsilon^r_{\unit}\cdot\id)}", from=8-3, to=9-3]
	\arrow["{R(\id\cdot UR(\varepsilon^r_{\unit}))}", from=8-3, to=8-5]
	\arrow["{R(U_2^{-1})}"{description}, from=6-3, to=8-5]
	\arrow["{(h^l_{\unit,R(\unit)R(\unit)})^{-1}}"{description}, from=6-3, to=4-5]
	\arrow["{\id\cdot \id\cdot m}", from=4-3, to=4-5]
	\arrow["{m\cdot \id}", from=4-5, to=4-6]
	\arrow[Rightarrow, no head, from=4-6, to=6-6]
	\arrow["{R(\varepsilon^r_{\unit})\cdot\id}"', from=6-5, to=6-6]
	\arrow["{\id\cdot \Delta}", from=2-5, to=4-5]
	\arrow["{\id\cdot m}"', from=2-3, to=2-5]
	\arrow[Rightarrow, no head, from=1-5, to=2-5]
	\arrow["{RU(h^l_{\unit,RUR(\unit)})}"', from=8-1, to=9-1]
	\arrow[""{name=1, anchor=center, inner sep=0}, "{R(U_2^{-1})}", from=8-1, to=8-3]
	\arrow[Rightarrow, no head, from=6-1, to=8-1]
	\arrow["{h^l_{\unit,R(\unit)}}"{description}, from=6-6, to=8-6]
	\arrow["{R(\varepsilon^r_{\unit}\cdot\id)}", from=8-5, to=8-6]
	\arrow[Rightarrow, no head, from=8-6, to=9-6]
	\arrow["{(\ref{lem:operatorCoHopf}(c))}"{description}, curve={height=-18pt}, draw=none, from=1-1, to=1-6]
	\arrow["{(N)}"{description}, draw=none, from=1-1, to=2-3]
	\arrow["{(\diamondsuit)}"{description}, draw=none, from=1-3, to=2-5]
	\arrow["{(\heartsuit)}"{description}, draw=none, from=1-5, to=4-6]
	\arrow["{(\heartsuit)}"{description}, draw=none, from=2-3, to=4-5]
	\arrow["{(N)}"{description}, draw=none, from=2-1, to=4-3]
	\arrow["{(N)}"{description}, draw=none, from=4-3, to=6-3]
	\arrow["{(\ref{lem:operatorCoHopf}(c))}"{description}, draw=none, from=4-5, to=6-6]
	\arrow["{(N)}"{description}, draw=none, from=6-5, to=8-6]
	\arrow["{(N)}"{description}, draw=none, from=8-3, to=9-6]
	\arrow["{(\ref{lem:operatorCoHopf}(b))}"{description}, draw=none, from=8-1, to=9-3]
	\arrow["{(\ref{lem:operatorCoHopf}(d))}"{description}, draw=none, from=4-5, to=8-3]
	\arrow["{(\ref{lem:operatorCoHopf}(c))}"{description}, draw=none, from=6-1, to=0]
	\arrow["{(N)}"{description}, draw=none, from=6-3, to=1]
\end{tikzcd}\]
}
Here, the squares marked $(\heartsuit)$ commute because $m$ and $\Delta$ satisfy the Frobenius relation. The square marked $(\diamondsuit)$ is the unitality relation of the algebra $(R(\unit),m,u)$.
Using (\ref{eq:etaL}), the bottom path of this above diagram equals $R(\eta^l_X \circ \varepsilon^r_X)$. Also, the top path reads $h^l_{\unit,R(\unit)}\circ \Delta \circ R(\varepsilon^r_X)$. Thus, we get that $R^2(\unit,\unit)=\Delta$.

Finally, as $R_2(\unit,\unit)$ and $\Delta =R_2(\unit,\unit)$ satisfy the Frobenius algebra axiom, we have that $((R(\unit),\sigma),R_2(\unit,\unit),R_0,\Delta,\nu)$ is a Frobenius algebra in $\Z(\C)$.
\end{proof}


\subsection{Separable and special Frobenius case}\label{subsec:3.spesep}
Now consider the following result.

\begin{theorem}\label{thm:speFrobcoHopf}
Assume that Condition~\ref{cond:main} is satisfied. Then, $(R(\unit),m,u,\allowbreak \Delta,\nu)$ is a separable (resp., special) Frobenius algebra in $\C$ if and only if $R$ is a separable (resp., special) Frobenius monoidal functor.
\end{theorem}

\begin{proof}
($\Rightarrow$): Suppose that $R(\unit)$ is a separable Frobenius algebra. To start, consider the following commutative diagram. 
{\small
\begin{equation*}
\begin{gathered}
\xymatrix@R=3.5pc@C=2.2pc{
    & & R(X)\cdot R(Y) \ar[r]^-{\eta^r_{R(X)\cdot R(Y)}} \ar@{}[d]|-(.5){(N)}  & RU(R(X) \cdot R(Y)) \ar[d]^{R(U_2^{-1})}    
    \\
    R(X\cdot Y) \ar[r]^-{R(\id\cdot\eta^l_Y)} & 
    R(X\cdot UR(Y)) \ar[ru]^{(h^l_{X,R(Y)})^{-1}} \ar[r]^-{\eta^r_{X\cdot UR(Y)}} \ar[rd]_{\id_{R(X\cdot UR(Y))}} & 
    RUR(X\cdot UR(Y)) \ar[d]|-{R(\varepsilon^r_{R(X\cdot UR(Y))})} \ar@{}[ld]|-(.25){(\ref{eq:snake})} \ar[ru]|-{RU((h^l_{X,R(Y)})^{-1})} \ar@{}[r]|-(.5){(\ref{lem:operatorCoHopf}(b))} & 
    R(UR(X)\cdot UR(Y))  \ar[ld]|-{R(\varepsilon^r_X\cdot \id_{UR(Y)})} \ar[d]^{R(\varepsilon^r_X \cdot \varepsilon^r_Y)} 
    \\
    & & R(X\cdot UR(Y)) \ar[r]_-{R(\id_X\cdot \varepsilon^r_Y)} & R(X\cdot Y) \ar@{}[ul]|-(.25){(N)}   
}
\end{gathered}
\end{equation*}
}
Using Theorems~\ref{thm:docadj}(a) and \ref{thm:frobcoHopf}(a) we get that the compositions along the top of the diagram is equal to  $R_2(X,Y)\; R^2(X,Y)$. Thus, we get that
\[R_2(X,Y)\; R^2(X,Y) = R(\id_X\otimes \varepsilon^r_Y\; \eta^l_Y ).\] 
Now, observe that $\varepsilon^r_Y\;\eta^l_Y\;\varepsilon^r_Y $ is equal to
\begin{align*}
    \stackrel{(\textnormal{\ref{thm:frobcoHopf}}(b))}{=} \;&
    \varepsilon^r_Y \circ UR(\varepsilon^r_Y) \circ \varepsilon^r_{URUR(Y)} \circ U(h^l_{\unit,RUR(Y)}) \circ U(\id_{R(\unit)} \cdot h^l_{\unit,R(Y)}) \circ U(\Delta  u \cdot\id_{R(Y)} )  
    \\[-0.07cm]
    \stackrel{(\textnormal{\ref{eq:coequalizer}})}{=} \;\;\;&
    \varepsilon^r_Y \circ \varepsilon^r_{UR(Y)} \circ \varepsilon^r_{URUR(Y)} \circ U(h^l_{\unit,RUR(Y)}) \circ U(\id_{R(\unit)} \cdot h^l_{\unit,R(Y)}) \circ U(\Delta  u \cdot\id_{R(Y)} )  
    \\[-0.07cm]
    \stackrel{(\textnormal{\ref{eq:coequalizer}})}{=} \;\;\;&
    \varepsilon^r_Y \circ \varepsilon^r_{UR(Y)} \circ UR(\varepsilon^r_{UR(Y)}) \circ U(h^l_{\unit,RUR(Y)}) \circ U(\id_{R(\unit)} \cdot h^l_{\unit,R(Y)}) \circ U(\Delta u \cdot\id_{R(Y)} )  
    \\[-0.07cm]
    \stackrel{(\textnormal{\ref{eq:coequalizer}})}{=} \;\;\;&
    \varepsilon^r_Y \circ UR(\varepsilon^r_Y) \circ UR(\varepsilon^r_{UR(Y)}) \circ U(h^l_{\unit,RUR(Y)}) \circ U(\id_{R(\unit)} \cdot h^l_{\unit,R(Y)}) \circ U(\Delta  u \cdot\id_{R(Y)} )  
    \\
    \stackrel{}{=} \;\;\;\;\,& 
    \varepsilon^r_Y \circ UR(\varepsilon^r_Y) \circ U[R(\id_{\unit}\cdot\varepsilon^r_{UR(Y)})(h^l_{\unit,RUR(Y)})] \circ U(\id_{R(\unit)} \cdot h^l_{\unit,R(Y)}) \circ U(\Delta  u \cdot\id_{R(Y)} )  
    \\[-0.07cm]
    \stackrel{(\textnormal{\ref{lem:operatorCoHopf}}(c))}{=} \;&
    \varepsilon^r_Y \circ UR(\varepsilon^r_Y) \circ U(R_2(\unit,UR(Y))) \circ U(\id_{R(\unit)} \cdot h^l_{\unit,R(Y)}) \circ U(\Delta  u \cdot\id_{R(Y)} )  
    \\[-0.07cm]
    \stackrel{(\textnormal{\ref{lem:operatorCoHopf}}(a))}{=}\; &
    \varepsilon^r_Y \circ UR(\varepsilon^r_Y) \circ U(h^l_{\unit,R(Y)}) \circ U(R_2(\unit,\unit)\cdot \id_{R(Y)}) \circ U(\Delta  u \cdot\id_{R(Y)} )  
    \\[-0.07cm]
    \stackrel{(\textnormal{\ref{lem:operatorCoHopf}}(c))}{=}\; &
    \varepsilon^r_Y \circ U(R_2(\unit,Y)) \circ U(R_2(\unit,\unit)\cdot \id_{R(Y)}) \circ U(\Delta  u \cdot\id_{R(Y)} )  
    \\
    \stackrel{}{=} \;\;\;\;\,& 
    \varepsilon^r_Y \circ U(R_2(\unit,Y)) \circ U((R_2(\unit,\unit)\circ \Delta)\cdot \id_{R(Y)}) \circ U(u \cdot \id_{R(Y)}) 
    \\[-0.07cm]
    \stackrel{(\diamondsuit)}{=} \;\;\;\, &
    \varepsilon^r_Y \circ U(R_2(\unit,Y))\circ U(u \cdot \id_{R(Y)}) 
    \\[-0.07cm]
    \stackrel{u=R_0}{=} \;\;& 
    \varepsilon^r_Y \circ U(R_2(\unit,Y)\circ (R_0\cdot \id_{R(Y)}))\hspace{0.65cm} \stackrel{(\spadesuit)}{=} \hspace{0.65cm}\varepsilon^r_Y    .
\end{align*}
Here, the equality $(\diamondsuit)$ holds because $R(\unit)$ is separable and the equality $(\spadesuit)$ holds because $R$ is a monoidal functor.
By Lemma~\ref{lem:premonadic}, $\varepsilon^r_Y$ is an coequalizer, and therefore, it is epic. Hence, we get that $\varepsilon^r_Y \; \eta^l_Y = \id_Y$, thereby proving that $R_2(X,Y)\; R^2(X,Y) = \id_{R(X\otimes Y)}$. Hence, $R$ is separable Frobenius.

Lastly, observe that $R^0\;R_0 = \nu\; u = \id_{\unit}$ is equal to the identity map on $\unit$ if and only if $R(\unit)$ is special Frobenius. Hence, this direction of the proof is finished.

\noindent
($\Leftarrow$): As the unit object is separable (resp., special) Frobenius, by Lemma~\ref{lem:mon-preserve}(b), it follows that $R(\unit)$ is separable (resp., special) Frobenius.
\end{proof}


\subsection{Pivotal case}\label{subsec:3.pivotal}
Recall the definitions of pivotal categories and functors from Section~\ref{subsec:duality}. Now, consider the following result.

\begin{proposition}\label{prop:piv}
Let $\C$ be a monoidal category.
\begin{enumerate}
    \item[\upshape(a)] Suppose that $((A,\sigma),m,u,\Delta,\nu)$ is a Frobenius algebra in $\Z(\C)$, then $F=A\otimes - :\C\rightarrow\C$ is a Frobenius monoidal functor with structure maps:
    {\small
    \begin{equation}
    \begin{split}
        F_2(X,Y):=(m\otimes \id_X\otimes \id_Y) (\id_{A}\otimes \sigma_X \otimes \id_Y)  & :A\otimes X \otimes A \otimes Y\rightarrow A \otimes X\otimes Y \\
        F^2(X,Y):= (\id_{A} \otimes \sigma_X^{-1} \otimes \id_Y )  (\Delta\otimes\id_X \otimes\id_Y) & : A \otimes X\otimes Y \rightarrow A\otimes X \otimes A \otimes Y\\
        F_0:= u:\unit\rightarrow A\hspace{2cm} & \hspace{1cm} F^0:= \nu: A \rightarrow \unit
    \end{split}
    \end{equation}
    }
    \item[\upshape(b)] If $\C$ is a pivotal category and $((A,\sigma),m,u,\Delta, \nu)$ is a symmetric Frobenius algebra in $\Z(\C)$, then $F = A\otimes - :\C\rightarrow\C$ is a pivotal functor with the above structure maps.
\end{enumerate}
\end{proposition}

\begin{proof}
The proof of (a) is straightforward, so we will only prove (b). 
Let $\fp$ denote the pivotal structure of $\C$. By the definition of a pivotal functor (Definition~\ref{defn:pivotalFunctor}), we need to prove that 
\begin{equation*}
    \xi^{A\otimes - }_X  \circ (\id_{A}\otimes \fp_X) = \fp_{A\otimes X}:A\otimes \lv\lv X \rightarrow \lv\lv(A\otimes X) = \lv\lv A \otimes \lv\lv X.
\end{equation*}  
where $\xi^{A\otimes - }_X 
\stackrel{(\textnormal{\ref{eq:xi}})}{=}  
\lv((\zeta^{A\otimes-}_X)^{-1}) \circ \zeta^{A\otimes-}_{\lv X}$. 
We first calculate that
\[
\begin{array}{rcl}
\zeta^{A\otimes-}_X 
& \stackrel{(\textnormal{\ref{eq:zeta}})}{=} 
& ([F^0\circ F(\ev_X) \circ F_2(\lv X, X)]  \otimes \id_{\lv F(X)})
\;(\id_{F(\lv X)}\otimes \coev_{F(X)}) 
\\
& \stackrel{(\spadesuit)}{=} &
(\nu\;m \otimes \ev_X \otimes \id_{\lv(A\otimes X)}  )(\id_A \otimes \sigma_{\lv X} \otimes \id_{X\otimes \lv(A\otimes X)})  (\id_{A\otimes \lv X}\otimes \coev_{A\otimes X})
\\
& \stackrel{(\diamondsuit)}{=} & ( \nu\; m\otimes \id_{\lv X \otimes \lv A }) (\id_A\otimes \sigma_{\lv X} \otimes \id_{\lv A}) (\id_{A\otimes \lv X}\otimes \coev_A)    .
\end{array}    
\]

\hspace{-0.1cm}Here, the equality $(\spadesuit)$ is obtained by plugging in the description of $F^2,F_2,F^0,F_0$ from part (a), and the equality $(\diamondsuit)$ using that $\coev_{A\cdot X} = (\id_A \otimes \coev_X \otimes \id_{\lv A})\coev_A$ and the snake relation. A similar calculation shows that 
\[ (\zeta^{A\otimes -}_X)^{-1} = \sigma_{\lv X}\; (\id_{\lv X} \otimes \ev_A \otimes \id_A)  (\id_{\lv X\cdot \lv A} \otimes u\; \Delta ) .\]

Using the descriptions of $\zeta^{A\otimes-}_X$ and $(\zeta^{A\otimes -}_X)^{-1} $ above, we get an expression for $\xi^{A\otimes - }_X $. Using the naturality of $\sigma$ and the snake relation one can simplify further to get that $\xi^{A\otimes - }_X =\kappa_A \otimes \id_{\lv\lv X}$, where $\kappa_A:A\rightarrow \lv\lv A$ is the following morphism
{\small
\[
\arraycolsep=1.4pt
\begin{array}{cl}
    = & (\nu m \otimes \id_{\lv\lv A}) (\id_A\otimes \ev_A\otimes\id_A\otimes \id_{\lv\lv A}) (\id_A\otimes \id_{\lv A} \otimes \Delta u \otimes \id_{\lv\lv A}) (\id_A\otimes \coev_{\lv A}) 
    \\

    \stackrel{(\textnormal{\ref{lem:symFrobalternate}})}{=} & (\nu m \otimes \id_{\lv\lv A}) (\id_A \otimes \id_A \otimes  \ev_{\lv A}(\fp_A\otimes \id_{\lv A}) \otimes \id_{\lv\lv A})  (\id_A\otimes \Delta  u \otimes \coev_{\lv A}) 
    \\
    \stackrel{(N)}{=} & (\ev_{\lv A} \otimes \id_{\lv\lv A}) (\id_{\lv\lv A}\otimes \coev_{\lv A})\fp_A (\nu  m\otimes \id_A)(\id_A\otimes \Delta  u) 
    \\
    \stackrel{(\heartsuit)}{=} & (\ev_{\lv A} \otimes \id_{\lv\lv A}) (\id_{\lv\lv A}\otimes \coev_{\lv A})\fp_{ A}  \hspace{0.5cm} = \;\; \fp_{ A} . 
\end{array}
\]   
}
Here, the equality $(\heartsuit)$ holds because $A$ is a Frobenius algebra. Finally,
\begin{equation*}
    \xi^{A\otimes - }_X  \circ (\id_{A}\otimes \fp_X) = (\fp_A\otimes\id_{\lv\lv X}) (\id_A\otimes \fp_X) = (\fp_A\otimes \fp_X) \stackrel{}{=} \fp_{A\otimes X}.
\end{equation*}
Hence, the proof is finished.
\end{proof}

Now, we are ready to prove the main result of this section.

\begin{theorem}\label{thm:pivotalcoHopf}
Suppose that Condition~\ref{cond:main} is satisfied with $U:\C \rightarrow \D$  a pivotal functor (thus, $\C$ is pivotal). Then, the algebra $(R(\unit),R_2(\unit,\unit),R_0)$ is a symmetric Frobenius algebra in $\C$ if and only if $R$ is a pivotal Frobenius monoidal functor.
\end{theorem}
\begin{proof}
We will prove the forward direction; the converse follows by Lemma~\ref{lem:mon-preserve}(c). 

Suppose that $(R(\unit),R_2(\unit,\unit),R_0)$ is symmetric Frobenius algebra in $\C$ with comultiplication $\Delta$ and counit $\nu$.
Then, by Proposition~\ref{prop:R1FrobinZC}, $((R(\unit),\sigma),R_2(\unit,\unit),R_0,\Delta,\nu)$ is a Frobenius algebra in $\Z(\C)$. Further, as the pivotal structure of $\Z(\C)$ is the same as that of $\C$, $R(\unit)$ being a symmetric Frobenius algebra in $\C$ implies that it is a symmetric Frobenius in $\Z(\C)$.
Hence, by Proposition~\ref{prop:piv}(b), $R(\unit)\otimes -$ is a pivotal functor. Since $U\dashv R$ is a coHopf adjunction by assumption, we get that $h^l_{\unit,X}: R(\unit)\otimes X\rightarrow R(\unit\otimes U(X)) = RU(X)$ is a monoidal natural isomorphism between the functors $R(\unit)\otimes - $ and $RU$. Therefore, we get that $RU$ is pivotal. Now, observe that
\[
\begin{array}{rl}
    \xi^R_{X} \circ R(\fp^{\D}_X) \circ R(\varepsilon^r_X) & \stackrel{(N)}{=} \xi^R_{X} \circ R(\lv\lv \varepsilon^r_X) \circ R(\fp^{\D}_{UR(X)}) \\
    & \stackrel{(N)}{=} \lv\lv R(\varepsilon^r_X) \circ \xi^R_{UR(X)} \circ R(\fp^{\D}_{UR(X)}) \\
    & \stackrel{(\textnormal{\ref{eq:pivotalFunctor}})}{=} \lv\lv R(\varepsilon^r_X) \circ \xi^R_{UR(X)} \circ R(\xi^U_{R(X)}) \circ RU(\fp^{\C}_{R(X)}) \\
    & \stackrel{(\textnormal{\ref{lem:pivShimizuA}})}{=} \lv\lv R(\varepsilon^r_X) \circ \xi^{RU}_{R(X)} \circ RU(\fp^{\C}_{R(X)}) \\
    & \stackrel{(\textnormal{\ref{eq:pivotalFunctor}})}{=} \lv\lv R(\varepsilon^r_X) \circ \fp^{\C}_{RUR(X)} \\
    & \stackrel{(N)}{=} \fp^{\C}_{R(X)} \circ  R(\varepsilon^r_X).
\end{array}
\]
As $\varepsilon^r_X$ is epic and $R$ is exact, $R(\varepsilon^r_X)$ is epic. Consequently, $\xi^R_{X} \circ R(\fp^{\D}_X) = \fp^{\C}_{R(X)} $, thereby proving that $R$ is pivotal.
\end{proof}


\subsection{Ribbon case}\label{subsec:3.ribbon}
In this section, we equip our categories with braidings and strengthen the results obtained in previous sections. 

\begin{lemma}\label{lem:adjointBraided}
Let $U:(\C,c)\rightarrow(\D,d)$ be a braided strong monoidal functor between braided monoidal categories. Then, $U^{\ra}$ is braided and $U^{\la}$ is cobraided.
\end{lemma}
\begin{proof}
Let $R:=U^{\ra}$. Observe that $R(d_{X,Y})\circ R_2(X,Y)$ is
\[ 
\begin{array}[]{cl}
    \stackrel{(\textnormal{\ref{thm:docadj}}(a))}{=} & 
    R(d_{X,Y})\circ R(\varepsilon_X\cdot \varepsilon_Y) \circ R(U_2^{-1}(R(X),R(Y))) \circ \eta_{R(X)\cdot R(Y)} \\
    \stackrel{(N)}{=} &
    R(\varepsilon_Y \cdot \varepsilon_X) \circ R(d_{UR(X),UR(Y)})  \circ R(U_2^{-1}(R(X),R(Y))) \circ \eta_{R(X)\cdot R(Y)} \\
    \stackrel{(\diamondsuit)}{=} &
    R(\varepsilon_Y \cdot \varepsilon_X) \circ R(U_2^{-1}(R(Y),R(X))) \circ RU(c_{R(X),R(Y)}) \circ \eta_{R(X)\cdot R(Y)}  \\ 
    \stackrel{(N)}{=} &
    R(\varepsilon_Y \cdot \varepsilon_X) \circ R(U_2^{-1}(R(Y),R(X))) \circ \eta_{R(Y)\cdot R(X)} \circ c_{R(X),R(Y)} \\
    \stackrel{(\textnormal{\ref{thm:docadj}}(a))}{=} & 
    R_2(Y,X) \circ c_{R(X),R(Y)} .
\end{array}
\]
Here, the equality $(\diamondsuit)$ holds because $U$ is braided. Thus, $R$ is braided. 

By Lemma~\ref{lem:coBraid}, we know that $U$ is cobraided. Hence, $U^{\op}$ is a braided strong monoidal functor.
Then, the claim about $L:=U^{\la}$ follows by applying the above result to the adjunction $U^{\op}\dashv L^{\op}$.
\end{proof}

\begin{theorem}\label{thm:ribboncoHopf}
Let $U:\C\rightarrow\D$ be a ribbon functor between ribbon categories satisfying Condition~\ref{cond:main}. Then $R(\unit)$ is a symmetric Frobenius algebra in $\C$ if and only if $R$ is a ribbon Frobenius functor.
\end{theorem}
\begin{proof}
By Theorem~\ref{thm:frobcoHopf}, $U^{\la}=U^{\ra}=R$. Since $U$ is braided, by Lemma~\ref{lem:adjointBraided}, we get that $U^{\la}=R$ is cobraided and $U^{\ra}=R$ is braided. Finally, as $R(\unit)$ is symmetric Frobenius, we obtain that $R$ is pivotal and Frobenius by Theorem~\ref{thm:pivotalcoHopf}. Now, by Proposition~\ref{prop:brPivotalFunctor}, $R$ is a ribbon Frobenius monoidal functor. The converse is straightforward.
\end{proof}


\section{Application to tensor categories}\label{sec:application}
Now, we apply the results of Section~\ref{sec:coHopf} to a functor $\Psi$ (studied in \cite{shimizu2020further}) from the Drinfeld center $\Z(\C)$ to certain category of endofunctors of a $\C$-module category $\M$, namely $\Rex_{\C}(\M)$. We recall background material in Section~\ref{subsec:4.background}, discuss the functor $\Psi$ in Section~\ref{subsec:4.main}, and provide a proof of our main application, Theorem~\ref{thm:intro2}, in Section~\ref{subsec:4.proof}.

\subsection{Background}\label{subsec:4.background}
We refer the reader to \cite{etingof2016tensor} for the following material.
Let $\C$ be a rigid monoidal category. If further, $\C$ is $\kk$-linear and finite abelian, the unit object is simple, and the tensor product $\otimes$ is $\kk$-bilinear, we call it a \textit{finite tensor category}. 

\subsubsection*{Module categories}
A \textit{(finite) left $\C$-module category} is a (finite abelian) category $\M$ equipped with a bifunctor $\tr:\C \times \M \rightarrow \M$ that is $\kk$-linear and exact in first variable, and natural isomorphisms
\begin{equation*}
    (X\otimes Y)\tr M \cong X\tr (Y\tr M), \hspace{0.75cm} \unit \tr M \cong M \hspace{1cm} (X,Y\in\C,\; M\in \M)
\end{equation*} 
satisfying certain coherence conditions. By a variant of Mac Lane's coherence theorem, we can (and will) assume that the above isomorphisms are identity maps. We call a finite left $\C$-module category \textit{exact} if for any $M\in\M$ and any projective object $P\in\C$, $P\tr M\in \M$ is projective. If $\C$ is a (pivotal) finite tensor category, then $\Z(\C)$ is a braided (pivotal) finite tensor category and $U_{\C}$ is a $\kk$-linear, exact (and pivotal) functor.

A left $\C$-module functor is a pair $(F,s)$ where $F:(\M,\tr_{\M})\rightarrow (\N,\tr_{\N})$ is a functor between module categories and $s=\{s_{X,M}: F(X\tr_{\M}M) \xrightarrow{\sim} X\tr_{\N}F(M) \}$ is a natural isomorphism satisfying certain conditions. 
Given a left $\C$-module category $\M$, we will use the notation $\Rex_{\C}(\M)$ to denote the category with objects as right exact, left $\C$-module endofunctors of $\M$ and morphisms as $\C$-module natural transformations. 
A \textit{relative Serre functor} \cite{schaumann2015pivotal,fuchs2020eilenberg} is a functor $\dS:\M\rightarrow \M $ equipped with a certain natural isomorphism.


\subsection{The functor $\Psi$}\label{subsec:4.main}
Consider the following terminology.

\begin{definition}\cite[Section~3.6]{shimizu2020further}\label{defn:Psi}
    Let $\C$ be a finite tensor category and $\M$ a left $\C$-module category. Consider the functor below
    \begin{equation}\label{eq:4Psi}
    \Psi:\Z(\C) \rightarrow \Rex_{\C}(\M), \hspace{1cm} (X,\sigma) \mapsto  (X\tr -,s^{\sigma}),
    \end{equation} 
    where the left $\C$-module structure of $X\tr -$ is 
    \[s^{\sigma}_{Y,M}: Y\tr (X\tr(M)) = (Y\otimes X)\tr M \xrightarrow{\sigma_Y\tr \id_M} (X\otimes Y)\tr M = X\tr(Y\tr M).\]
\end{definition}

In order to better understand $\Psi$, we define the following functors.

\begin{definition}\label{defn:PsiStructureFunctors}
Set $\D:=\Rex_{\C}(\M)$.
\begin{itemize}
    \item For any monoidal category $\C$, consider the forgetful functor
    \begin{equation}\label{eq:4forgetful}
        U'_{\C}:\Z(\C)^{\mir}\rightarrow \C \hspace{1cm} (X,\sigma)\mapsto X.
    \end{equation}
    As a monoidal functor, $U_{\C}$ and $U'_{\C}$ are identical. Thus, many facts about $U_{\C}$ also hold true for $U'_{\C}$. In particular, $U'_{\C}$ is a strong monoidal functor. When $\C$ is a finite tensor category, $U'_{\C}$ admits a right adjoint $R_{\C}$.

    \item Schauenburg's \cite[Theorem~3.3]{schauenburg2001monoidal} established the following braided equivalence between the Drinfeld centers of $\C$ and $\D^{\rev}$. 
    \begin{equation}\label{eq:4Schauenburg}
        \Theta_{\M}:\Z(\C)\xrightarrow{\sim} \Z(\D^{\rev}), \;\; (X,\sigma)\mapsto ((X\tr -,s^{\sigma}),\Sigma ).
    \end{equation}
    For the definition of $\Sigma$, see \cite[Section~3.7]{shimizu2020further}.
    
    \item For any monoidal category $\C$, by \cite[Exercise~8.5.2]{etingof2016tensor}, we have the following braided equivalence 
    \begin{equation}\label{eq:4mirror}
        \Omega_{\C}: \Z(\C^{\rev})\cong \Z(\C)^{\mir}, \hspace{1cm} (X,\sigma)\mapsto (X,\sigma^{-1})  .
    \end{equation}
\end{itemize}
\end{definition}

Now consider the following result.
\begin{lemma}\cite[Theorem~3.14]{shimizu2020further}\label{lem:PsiStructure}
The functor $\Psi$ is equal to the composition $U_{\D}'\circ \Omega_{\D} \circ \Theta_{\M}$. Further, $\Psi$ is an exact, strong monoidal functor.
\end{lemma}
\begin{proof}
Observe that 
{\small
\[
\begin{array}{rclcl}
    U'_{\D}\circ \Omega_{\D} \circ \Theta_{\M}(X,\sigma) 
    & \stackrel{(\textnormal{\ref{eq:4Schauenburg}})}{=} & U'_{\D}\circ \Omega_{\D} ((X\tr - ,s^{\sigma}),\Sigma) 
    & \stackrel{(\textnormal{\ref{eq:4mirror}})}{=} & U'_{\D} ((X\tr - ,s^{\sigma}),\Sigma^{-1}) \\
    & \stackrel{(\textnormal{\ref{eq:4forgetful}})}{=} & (X\tr -, s^{\sigma})
    & \stackrel{(\textnormal{\ref{eq:4Psi}})}{=} & \Psi(X,\sigma) .
\end{array}
\]
}
Since $U_{\D}'$, $\Omega_{\D}$ and $\Theta_{\M}$ are each strong monoidal, $\Psi$ is a strong monoidal functor. Furthermore, by \cite[Theorem~3.11]{shimizu2020further}, $\Psi$ admits a right adjoint $\Psi^{\ra}$. Since the categories $\Z(\C)$ and $\Rex_{\C}(\M)$ are rigid monoidal, the existence of a right adjoint implies the existence of a left adjoint, see \cite[\S 2.3]{bruguieres2011exact}. Finally, as $\Z(\C)$ and $\Rex_{\C}(\M)$ are finite categories, the existence of adjoints implies that $\Psi$ is exact. 
\end{proof}

Now, let $\C$ be a pivotal tensor category with pivotal structure $\mathfrak{p}:\id_{\C}\xrightarrow{\cong} \lv\lv(-)$, and let $\M$ be an exact left $\C$-module category. Then, by \cite[Lemma~3.3]{shimizu2019relative}, a (right) relative Serre functor $\dS$ of $\M$ exists. Further, $\dS$ is a left $\C$-module functor. 
In this case, a \textit{pivotal structure} on $\M$ is a left $\C$-module natural isomorphism $\widetilde{\fp} :\id_{\M}\rightarrow \dS$, and a \textit{pivotal left $\C$-module category} is an exact left $\C$-module category equipped with a pivotal structure. 

For $\M$ a pivotal left $\C$-module category, the category $\D:=\Rex_{\C}(\M)$ is a pivotal monoidal category \cite[Theorem~3.13]{shimizu2019relative}. In fact, we get the result below.

\begin{lemma}\label{lem:psiPivotal}
If $\C$ is a pivotal finite tensor category and $\M$ is a pivotal left $\C$-module category, then $\Psi$ is a pivotal functor.
\end{lemma}
\begin{proof}
By Lemma~\ref{lem:PsiStructure}, $\Psi = U_{\D}' \circ \Omega_{\D} \circ \Theta_{\M}$. By \cite[Proposition~5.14]{spherical2022}, $\Theta_{\M}$ is a pivotal functor. It is straightforward to check that $\Omega_{\D}$ is a pivotal functor.  Also, for any pivotal monoidal category $\D$, the forgetful functor $\Z(\D)\rightarrow \D$ is pivotal (see e.g. \cite[Section~5.2.2]{turaev2017monoidal}). Thus, $U'_{\D}$ is pivotal. By Lemma~\ref{lem:pivShimizuB}, the composition of pivotal functors is pivotal. Hence, we conclude that $\Psi$ is pivotal.
\end{proof}

Finally, we prove our main result of this section, and last result of this work.


\subsection{Proof of Theorem~\ref{thm:intro2}}\label{subsec:4.proof}
We know that $\Psi$ is a strong monoidal functor between abelian monoidal categories. Since it is an exact functor (by Lemma~\ref{lem:PsiStructure}) between finite abelian categories, it admits a right adjoint $\Psi^{\ra}$. By Theorem~\ref{thm:docadj}(a), $\Psi\dashv \Psi^{\ra}$ is a comonoidal adjunction. We will show that the adjunction $\Psi\dashv\Psi^{\ra}$ satisfies Condition~\ref{cond:main}.

\noindent
\upshape{(a)}: As $\Z(\C)$, $\Rex_{\C}(\M)$ are rigid, by Remark~\ref{rem:coHopf}, $\Psi\dashv \Psi^{\ra}$ is a coHopf adjunction.

\noindent
\upshape{(b)}: Using Lemma~\ref{lem:PsiStructure}, $\Psi^{\ra} = \Theta_{\M}^{\ra}\circ \Omega_{\D}^{\ra}\circ (U'_{\D})^{\ra}$. As $\Theta_{\M}, \, \Omega_{\D}$ are category equivalences, their right adjoint are also category equivalences, and in particular, exact. Lastly, using \cite[Proposition~3.39(ii)]{etingof2004finite}, $(U'_{\D})^{\ra}$ is exact. Thus, $\Psi^{\ra}$ is exact.
 
\noindent
\upshape{(c)}: Since $\M$ is indecomposable, $\D=\Rex_{\C}(\M)$ satisfies $\End_{\D}(\id_{\M})\cong \kk$. Thus, by \cite[Corollary~5.9]{shimizu2016unimodular}, we get that the functor $(U'_{\D})^{\ra}$ is faithful. Since $\Theta_{\M}^{\ra}$ and $\Omega_{\D}$ are category equivalences, we get that $\Psi^{\ra} = \Theta_{\M}^{\ra}\circ \Omega_{\D}^{\ra}\circ U_{\D}^{'\ra}$ is faithful.

Now, part \upshape{(i)} follows from Theorem~\ref{thm:frobcoHopf} and Theorem~\ref{thm:speFrobcoHopf}. Further, when $\C$ is pivotal and $\M$ is a pivotal left $\C$-module category, by Lemma~\ref{lem:psiPivotal}, $\Psi$ is pivotal. Hence, part \upshape{(ii)} follows from Theorem~\ref{thm:pivotalcoHopf}.      \qed

\bibliographystyle{alpha}
\bibliography{references}

\end{document}